\pdfoutput=1
\RequirePackage{ifpdf}
\ifpdf 
\documentclass[pdftex]{sigma}
\else
\documentclass{sigma}
\fi

\usepackage{wasysym}

\newcommand{\nc}{\newcommand}
\newcommand{\lbutcher}{{\raise 3.9pt\hbox{$\circ$}}\hskip -1.9pt{\scriptstyle \searrow}}

\numberwithin{equation}{section}

\newtheorem{Theorem}{Theorem}[section]
\newtheorem*{Theorem*}{Theorem}

\newtheorem{Lemma}[Theorem]{Lemma}
\newtheorem{Proposition}[Theorem]{Proposition}
 { \theoremstyle{definition}

\newtheorem{Remark}[Theorem]{Remark} }

\nc{\vep}{\varepsilon} \nc{\ep}{\epsilon}
\nc{\sigmat}{\widetilde\sigma}
\nc{\ostar}{\overline{*}}

\nc{\grad}[1]{^{({#1})}}
\nc{\fil}[1]{_{#1}}

\nc{\BA}{{\Bbb A}} \nc{\CC}{{\Bbb C}} \nc{\DD}{{\Bbb D}}
\nc{\EE}{{\Bbb E}} \nc{\FF}{{\Bbb F}} \nc{\GG}{{\Bbb G}}
\nc{\HH}{{\Bbb H}} \nc{\LL}{{\Bbb L}} \nc{\NN}{{\Bbb N}}
\nc{\PP}{{\Bbb P}} \nc{\QQ}{{\Bbb Q}} \nc{\RR}{{\Bbb R}}
\nc{\TT}{{\Bbb T}} \nc{\VV}{{\Bbb V}} \nc{\ZZ}{{\Bbb Z}}
\nc{\Cal}[1]{{\mathcal {#1}}}
\nc{\mop}[1]{\mathop{\hbox {\rm #1} }}
\nc{\mopl}[1]{\mathop{\hbox {\rm #1} }\limits}
\nc{\frakg}{{\frak g}}
\nc{\g}[1]{{\frak {#1}}}
\nc{\wt}{\widetilde}
\nc{\wh}{\widehat}
\nc{\un}{\hbox{\bf 1}}

\nc\fleche[1]{\mathop{\hbox to #1 mm{\rightarrowfill}}\limits}
\def\semi{\mathrel{\times}\kern -.85pt\joinrel\mathrel{\raise
1.4pt\hbox{${\scriptscriptstyle |}$}}}

\def\fleche#1{\mathop{\hbox to #1 mm{\rightarrowfill}}\limits}
\def\gfleche#1{\mathop{\hbox to #1 mm{\leftarrowfill}}\limits}
\def\inj#1{\mathop{\hbox to #1 mm{$\lhook\joinrel$\rightarrowfill}}\limits}
\def\ginj#1{\mathop{\hbox to #1 mm{\leftarrowfill$\joinrel\rhook$}}\limits}
\def\surj#1{\mathop{\hbox to #1 mm{\rightarrowfill\hskip 2pt\llap{$\rightarrow$}}}\limits}
\def\gsurj#1{\mathop{\hbox to #1 mm{\rlap{$\leftarrow$}\hskip 2pt
\leftarrowfill}}\limits}
\def \restr#1{\mathstrut_{\textstyle |}\raise-6pt\hbox{$\scriptstyle #1$}}
\def \srestr#1{\mathstrut_{\scriptstyle |}\hbox to
-1.5pt{}\raise-4pt\hbox{$\scriptscriptstyle #1$}}

\begin{document}
\allowdisplaybreaks

\newcommand{\arXivNumber}{2108.11103}

\renewcommand{\PaperNumber}{023}

\FirstPageHeading

\ShortArticleName{Post-Lie Magnus Expansion and BCH-Recursion}

\ArticleName{Post-Lie Magnus Expansion and BCH-Recursion}

\Author{Mahdi J.~Hasan AL-KAABI~$^{\rm a}$, Kurusch EBRAHIMI-FARD~$^{\rm b}$ and Dominique MANCHON~$^{\rm c}$}

\AuthorNameForHeading{M.J.H. Al-Kaabi, K.~Ebrahimi-Fard and D.~Manchon}

\Address{$^{\rm a)}$~Mathematics Department, College of Science, Mustansiriyah University,\\
\hphantom{$^{\rm a)}$}~Palestine Street, P.O.~Box 14022, Baghdad, Iraq}
\EmailD{\href{mailto:mahdi.alkaabi@uomustansiriyah.edu.iq}{mahdi.alkaabi@uomustansiriyah.edu.iq}}

\Address{$^{\rm b)}$~Department of Mathematical Sciences, Norwegian University of Science and Technology,\\
\hphantom{$^{\rm b)}$}~Trondheim, Norway}
\EmailD{\href{mailto:kurusch.ebrahimi-fard@ntnu.no}{kurusch.ebrahimi-fard@ntnu.no}}

 \Address{$^{\rm c)}$~Laboratoire de Math\'ematiques Blaise Pascal, CNRS et Universit\'e Clermont-Auvergne\\
\hphantom{$^{\rm c)}$}~(UMR 6620), 3 place Vasar\'ely, CS 60026, F63178 Aubi\`ere, France}
\EmailD{\href{mailto:dominique.manchon@uca.fr}{dominique.manchon@uca.fr}}

\ArticleDates{Received August 26, 2021, in final form March 10, 2022; Published online March 23, 2022}

\Abstract{We identify the Baker--Campbell--Hausdorff recursion driven by a weight $\lambda=1$ Rota--Baxter operator with the Magnus expansion relative to the post-Lie structure naturally associated to the corresponding Rota--Baxter algebra. Post-Lie Magnus expansion and BCH-recursion are reviewed before the proof of the main result.}


\Classification{16T05; 16T10; 16T30; 17A30}\vspace{1mm}

\section{Introduction}

The present paper consists in a quick survey of post-Lie algebras, Baker--Campbell--Haudorff recursion, Rota--Baxter algebras and post-Lie Magnus expansion (Sections~\ref{s1}, \ref{s2} and \ref{s3}), followed by a new result in Section~\ref{s4}, which establishes an equality between two seemingly different formal series: the Baker--Campell--Hausdorff recursion in a weight-one Rota--Baxter algebra, and the post-Lie Magnus expansion relative to the associated post-Lie algebra structure described in~\cite{BGN10}. Our motivation comes from a result obtained in 2006 by the second and the third author together with Li Guo~\cite{KLM06}, identifying the BCH-recursion with the pre-Lie Magnus expansion in the weight-zero case.

Interest in Magnus-type expansions results from their appearance in the context of numerical integration methods for Lie group valued problems \cite{IserlesNorsett99}. In \cite{CEO20}, we studied the relation between the classical and the post-Lie Magnus expansions by looking at a non-autonomous matrix-valued initial value problem from the viewpoint of the theory of numerical Lie group integrators. Post-Lie algebras and the post-Lie Magnus expansion play a central part in the latter. In~this context post-Lie algebras characterise the relation between two Lie algebras (one coming from the Jacobi Lie bracket and the other from the torsion Lie bracket, in the context of a flat connection with constant torsion). This relation can be lifted to the level of the (completed) enveloping algebra (of the Lie algebra implied by the torsion Lie bracket). In~a~nut\-shell, the post-Lie Magnus expansion naturally appears in the context of backward error analysis for the Lie--Euler method. This is consistent with the fact that the pre-Lie Magnus expansion plays an analogous role with respect to backward error analysis for the Euler method.

The well-known Baker--Campbell--Hausdorff (BCH) formula ${\rm BCH}(x, y)$ is a formal power series, which lives in the completion of the free Lie algebra $\Cal{L}(x, y)$ generated (over a base field~$K$ of characteristic zero) by the two non-commutating variables $x$ and $y$. It is defined~by
\begin{displaymath}
\exp(x)\exp(y)= \exp\big( {\rm BCH}(x, y)\big)= \exp\big( x + y+ \wt{{\rm BCH}}(x, y)\big)
\end{displaymath}
or
\begin{displaymath}
{\rm BCH}(x, y)=\log\big( \exp(x)\exp(y)\big)= x + y + \wt{{\rm BCH}}(x, y).
\end{displaymath}
It plays a prominent role in modern mathematics \cite{AAFCCC020,BF12}.\footnote{The remainder Baker--Campbell--Hausdorff series, $\wt{{\rm BCH}}(x,y)$, is denoted by ${\rm BCH}(x,y)$ in \cite{KLM06}. We adopt here a more conventional notation.}

A fruitful connection between the BCH-series and the notion of Rota--Baxter algebra has been explored in \cite{KLK04, KLK05, KLM06}. The latter originated in the seminal 1960 article \cite{B60} by the American mathematician G.~Baxter, which in turn was motivated by F. Spitzer's 1956 article~\cite{Sp56}. Baxter's algebra was further developed foremost in the commutative realm in the 1960s and '70s by P.~Cartier, G.-C.~Rota and F.V.~Atkinson, among others, from algebraic, combinatorial and analytic viewpoints. We refer the reader to the review article \cite{EP19} as well as the monograph~\cite{Guo12} for details.

A weight-$\lambda$ Rota--Baxter operator on an associative $K$-algebra $\Cal{A}$ is a $K$-linear map $\Cal{R}$: \mbox{$\Cal{A} \longrightarrow \Cal{A}$}, satisfying the Rota--Baxter identity of weight $\lambda \in K$:
\begin{equation}
\label{RBidentity}
\Cal{R}(x)\Cal{R}(y)
= \Cal{R} \big(\Cal{R}(x)y + x \Cal{R}(y) + \lambda xy \big), \qquad x,y \in\Cal A.
\end{equation}
The pair $(\Cal A,\Cal R)$ is a weight $\lambda$ Rota--Baxter algebra.\footnote{The convention for the weight is with the opposite sign in \cite{KLM06}.} For example, the indefinite Riemann integral satisfies \eqref{RBidentity} when the weight $\lambda=0$ (integration by parts). The linear map $\widetilde{\Cal{R}} := - \lambda {\rm id}_{\Cal{A}} - \Cal{R}$ is also Rota--Baxter of weight $\lambda$, and satisfies together with $\Cal R$ the mixed identity
\begin{displaymath}
\Cal R(x)\widetilde{\Cal R}(y)
=\widetilde{\Cal R}\big(\Cal R(x)y\big)+\Cal R\big(x\widetilde{\Cal R}(y)\big), \qquad x,y \in\Cal A.
\end{displaymath}
Starting from a Rota--Baxter operator $\Cal R$ of weight $\lambda$, the BCH-recursion \cite{KLM06} is defined by
\begin{equation}\label{rec.1}
\chi_{\lambda}(a) := a + \frac{1}{\lambda} \widetilde{{\rm BCH}}\big(\Cal{R}\big(\chi_{\lambda}(a)\big),
\widetilde{\Cal{R}}\big(\chi_{\lambda}(a)\big)\big),\qquad a\in\Cal A.
\end{equation}

It lies at the heart of the solution of an exponential factorisation problem \cite{KLM06} and thereby permits the generalisation of a classical result for commutative Rota--Baxter algebras, known as Spitzer's identity \cite{Sp56}, to non-commutative Rota--Baxter algebras. The resulting non-com\-mu\-ta\-tive Spitzer identity says that for $a \in \Cal{A}$ the exponential
\begin{displaymath}
X:=\exp\bigg( \Cal{R}\bigg(\ \chi_{\lambda}\bigg( \frac{\log(1+t\lambda a)}{\lambda}\bigg)\bigg) \bigg)
\end{displaymath}
solves the fixed point equation
\begin{equation}
\label{FPE1}
X = 1 + t\Cal{R}(aX)
\end{equation}
 in the algebra $\Cal{A}[[t]]$ of formal series with coefficients in $\Cal A$. Here the formal parameter $t$ commutes with all elements in $\Cal{A}$.
More precisely, iterating the fixed point equation \eqref{FPE1} yields the rather non-trivial equality
\begin{displaymath}
1+ t\Cal{R}(a) + t^2\Cal{R}\big(a\Cal{R}(a)\big) + t^3\Cal{R}\big(a\Cal{R}\big(a\Cal{R}(a)\big)\big) + \cdots
= \exp\bigg( \Cal{R}\bigg(\ \chi_{\lambda}\bigg( \frac{\log(1+ t\lambda a)}{\lambda}\bigg)\bigg) \bigg).
\end{displaymath}
Thanks to the commuting parameter $t$, the last equality can be seen as between formal power series and therefore encompasses at each order a specific relation between coefficients. For instance, at order two, that is, comparing the coefficients of $t^2$, we have the identity
\begin{displaymath}
2 \Cal{R}\big(a\Cal{R}(a)\big) = \Cal{R}(a)\Cal{R}(a) - \Cal{R}\big([\Cal{R}(a),a] + \lambda a^2\big),
\end{displaymath}
which is easily verifiable in a Rota--Baxter algebra of weight~$\lambda$ by using the Rota--Baxter identity~\eqref{RBidentity} on the right-hand side. We note that the fixed point equation \eqref{FPE1} is reminiscent of the integral fixed point equation naturally associated to a linear matrix-valued initial value problem; the indefinite Riemann integral is a weight-zero Rota--Baxter map. Indeed, the series~\eqref{rec.1} turns out to be closely related to a well-known Lie algebra expansion due to W.~Mag\-nus~\cite{WM54}. This connection to the so-called Magnus expansion was studied in reference \cite{KLM06} in the case of the weight being zero ($\lambda=0$). The adequate algebraic setting is provided through the notion of pre-Lie algebra, which is naturally defined on any non-commutative Rota--Baxter algebra. In \cite{KM08} it was shown that the pre-Lie Magnus expansion can be expressed in terms of the BCH-recursion as follows
\begin{equation}
\label{preLieMag}
\Omega'_\rhd(a)
:= a + \sum_{n>0} \frac{B_n}{n!} L^{(n)}_{\rhd}\big[\Omega'_\rhd(a)\big](a)
= \chi_{\lambda}\bigg( \frac{\log(1+\lambda a)}{\lambda}\bigg).
\end{equation}
Here $B_n$ is the $n$-th Bernoulli number and $L_{\rhd}[x](y)=L^{(1)}_{\rhd}[x](y):=x \rhd y$ is the left-multiplication operator defined in terms of the aforementioned (left) pre-Lie product, denoted $\rhd$, on a non-commutative Rota--Baxter algebra. Note that the weight $\lambda$ is absorbed in the definition of the pre-Lie product. In the weight-zero case, \eqref{preLieMag} boils down to
\begin{equation}
\label{preLieMag-zero}
\Omega'_\rhd(a)
= \chi_{0}(a).
\end{equation}
In particular, for the indefinite Riemann integral, the pre-Lie product is defined for -- matrix-valued -- functions $A$, $B$ as $(A \rhd B)(t):=\big[\int_0^tA(s){\rm d}s,B(t)\big]$. When inserted in \eqref{preLieMag}, one recovers Magnus' original expansion \cite{WM54}.

Recall that any Rota--Baxter algebra with nonzero weight gives rise to a post-Lie algebra structure \cite{BGN10}. In this work, we describe a close relationship between the BCH-recursion \eqref{rec.1} in the nonzero weight case and the Magnus expansion in its post-Lie version \cite{ELM2015, EI2018, EIM2017,Men2020}. Our main result (Theorem \ref{main Theorem}) shows that the post-Lie Magnus expansion and the BCH-recursion in~\eqref{rec.1} coincide in the context of a Rota--Baxter algebra of weight 1 endowed with its naturally associated post-Lie structure. This is an extension to nonzero weight of one of the main results of \cite{KLM06}, resumed by \eqref{preLieMag-zero}, identifying the weight zero BCH-recursion with the pre-Lie Magnus expansion. The special role of weight one here simply comes from the definition of the post-Lie structure \eqref{pl1}, \eqref{pl2}, and any Rota--Baxter algebra with nonzero weight can be set to weight one by an appropriate rescaling of the Rota--Baxter operator.

We close this introduction by noting that the Magnus expansion, in its various forms (classical \cite{WM54,MP1970}, pre-Lie \cite{AG81, CP13, KM08} and post-Lie \cite{ELM2015, EI2018, EIM2017,Men2020}), has been studied in applied mathematics, control theory, physics and chemistry. See reference \cite{BCOR08} for details on the classical Magnus expansion in applied mathematics. The reader can also find a brief summary in the recent work~\cite{CEO20}.

This paper consists of four sections accompanied by two appendices. In Section~\ref{s1}, we review some basic topics related to post-Lie algebras and their universal enveloping algebras. The post-Lie structure defined on any Rota--Baxter algebra is recalled from~\cite{BGN10}. Section~\ref{s2} contains the description of the Baker--Campbell--Hausdorff recursion and its inverse, as well as their properties. Several important details on the post-Lie Magnus expansion and its inverse are included in Section~\ref{s3}. Section~\ref{s4} is the main part of this work, in which the identification of the post-Lie Magnus expansion with the BCH-recursion is proven. Finally, the two Appendices \ref{C.pl-Magnus} and \ref{C.ipl-Magnus} contain low-order computations of the post-Lie Magnus expansion and its inverse.

\section{Post-Lie algebras}\label{s1}

A post-Lie algebra is a Lie algebra $(\Cal{L},[\cdot\,,\cdot])$ together with a bilinear mapping $\rhd\colon \Cal{L} \times \Cal{L} \longrightarrow \Cal{L}$, which is compatible with the Lie bracket in the following sense
\begin{gather}\label{eq1}
x \rhd [y,z] = [x\rhd y,z]+ [y, x \rhd z],
\\
\label{eq2}
[x,y]\rhd z = a_{\rhd}(x,y,z)-a_{\rhd}(y,x,z),
\end{gather}
for any $x, y, z \in \Cal{L}$. Here, $a_{\rhd}(x,y,z)$ is the associator defined by
\begin{displaymath}
a_{\rhd}(x,y,z)=x \rhd (y \rhd z) - (x \rhd y) \rhd z.
\end{displaymath}
Any Lie algebra can be seen as a post-Lie algebra by setting the second product $\rhd$ to zero. Another possibility is to take for the second product $\rhd$ the opposite of the Lie bracket.

A (left) pre-Lie algebra is an abelian post-Lie algebra, i.e., a post-Lie algebra with Lie bracket set to zero. The defining relation is the left pre-Lie identity
\begin{equation}
\label{left-pl}
0=a_\rhd(x,y,z)-a_\rhd(y,x,z).
\end{equation}
We refer the reader to \cite{AL2011} for a short survey on pre-Lie algebras. The post-Lie operation $\rhd$ permits to produce two other operations:
\begin{gather*}
[\![x, y]\!] := x \rhd y - y \rhd x + [x, y],
\\
x \RHD y := x \rhd y + [x, y],
\end{gather*}
for all $x, y \in \Cal{L}$. From \eqref{eq1} and \eqref{eq2}, one can see that $\big(\Cal{L}, [\![\cdot\,,\cdot ]\!]\big)$ forms a Lie algebra, denoted $\tilde{\Cal L}$. In the case of an abelian post-Lie algebra, this amounts to Lie admissibility of pre-Lie algebras. The triple $\big(\Cal{L}, - [\cdot\,,\cdot ], \RHD\big)$ forms another post-Lie algebra \cite{CEO20, HA13} sharing the same double Lie bracket, i.e.,
\begin{displaymath}
[\![x, y]\!]
= x \rhd y - y \rhd x + [x, y]
=x \RHD y - y \RHD x - [x, y].
\end{displaymath}
For more details on post-Lie algebras, we refer to \cite{CEMK18,ELM2015,EI2018,HA13} and references therein.

\subsection{The universal enveloping algebra of a post-Lie algebra}

Inspired by the work of J.-M.~Oudom and D.~Guin in the pre-Lie context \cite{OG2008}, the authors in~\cite{ELM2015} consider the enveloping algebra $\big(\Cal{U}(\Cal{L}),\cdot\big)$ of the Lie algebra $\big(\Cal{L}, [\cdot\,,\cdot ]\big)$ underlying a post-Lie algebra $\big(\Cal{L}, [\cdot\,,\cdot ],\rhd\big)$. The post-Lie product $\rhd$ is then extended to $\Cal{L} \otimes \Cal{U}(\Cal{L}) \rightarrow \Cal{U}(\Cal{L})$ by requiring $x\rhd\un:=0$ and
\begin{displaymath}
x\rhd(x_1\cdots x_n) := \sum_{i=1}^{n}{x_1 \cdots x_{i-1}(x \rhd x_i)x_{i+1}\cdots x_n},
\end{displaymath}
for all $x, x_1, \ldots, x_n \in \Cal{L}$. Here, $\mathbf{1}$ denotes the unit in $\Cal{U}(\Cal{L})$. Recall that the enveloping alge\-bra~$\Cal{U}(\Cal{L})$ together with the product $\cdot$ and the unshuffle coproduct has the structure of a~non-commutative, co-commutative Hopf algebra. The unshuffle coproduct $\Delta$ is defined for all letters $x \in \Cal{L} \hookrightarrow \Cal{U}(\Cal{L})$, by $\Delta(x):= x \otimes \mathbf{1} + \mathbf{1} \otimes x$ and extended multiplicatively. We employ Sweedler's notation, $\Delta(X):= X_{(1)} \otimes X_{(2)}$, for the coproduct of any $X \in \Cal{U}(\Cal{L})$. The final defi\-ni\-tion of the extended post-Lie product on $\Cal{U}(\Cal{L})$, together with its properties, is given by the next two propositions.
\begin{Proposition}[{\cite[Proposition 3.1]{ELM2015}}]
There is a unique extension of the post-Lie product $\rhd$ from $\Cal{L}$ to $\Cal{U}(\Cal{L})$ satisfying:
\begin{gather*}
\mathbf{1} \rhd X = X,
\\
x X \rhd y = x \rhd (X \rhd y) - (x \rhd X) \rhd y,
\\
X \rhd Y Z = \big(X_{(1)} \rhd Y\big)\big(X_{(2)} \rhd Z\big),
\end{gather*}
for all $x, y \in \Cal{L}$, and $X, Y, Z \in \Cal{U}(\Cal{L})$.
\end{Proposition}

\begin{Proposition}[{\cite[Proposition 3.2]{ELM2015}}]
\label{prop:2}
The extended post-Lie product $\rhd$ on $\Cal{U}(\Cal{L})$ possesses the following properties:
\begin{gather*}
X \rhd \mathbf{1} = \epsilon(X),
\\
\epsilon(X \rhd Y) = \epsilon(X)\epsilon(Y),
\\
\Delta(X \rhd Y) = (X_{(1)} \rhd Y_{(1)}) \otimes (X_{(2)} \rhd Y_{(2)}),
\\
x X \rhd Y = x \rhd (X \rhd Y) - (x \rhd X) \rhd Y,
\\
X \rhd (Y \rhd Z) = (X_{(1)}\big(X_{(2)} \rhd Y\big)) \rhd Z,
\end{gather*}
for all $x \in \Cal{L}$ and $X, Y, Z \in \Cal{U}(\Cal{L})$, where $\epsilon\colon \Cal{U}(\Cal{L}) \rightarrow K$ is the counit map.
\end{Proposition}

From the last equality in Proposition \ref{prop:2}, an associative product, known as Grossman--Larson product, can be defined on $\Cal{U}(\Cal{L})$ as follows
\begin{equation}\label{ast}
X \ast Y := X_{(1)}\big(X_{(2)} \rhd Y\big),
\end{equation}
for all $X, Y \in \Cal{U}(\Cal{L})$. As a main example, for any $x\in \Cal{L}$ and $Y\in\mathcal U(\mathcal L)$, we find
\begin{equation}\label{GL-simple}
x \ast Y = x \rhd Y + xY,
\end{equation}
since any element of $\Cal L$ is primitive. The Grossman--Larson product~\eqref{ast} defines together with the coproduct $\Delta$ another structure of Hopf algebra on $\Cal{U}(\Cal{L})$. The corresponding antipode will be denoted by~$S_*$. The Hopf algebras $\big(\Cal{U}(\Cal{L}), \ast, \Delta \big)$ and $\big(\Cal{U}(\tilde{\Cal{L}}), \textbf{.}, \Delta \big)$ are isomorphic \cite[Section~3]{KLM06}, \cite[Section~2]{OG2008}.
\begin{Remark}
Conversely, the product of the enveloping algebra can be expressed in terms of the Grossman--Larson product and the unshuffle coproduct as follows
\begin{equation}\label{gl-rev}
XY=X_{(1)}*\big(S_*X_{(2)}\rhd Y\big).
\end{equation}
This is seen by plugging \eqref{ast} into the right-hand side of \eqref{gl-rev}.
\end{Remark}

\subsection{Free post-Lie algebras}

F.~Chapoton and M.~Livernet presented in \cite{CL01} the free pre-Lie algebra in terms of (non-planar) decorated rooted trees. Similarly, H.~Munthe-Kaas and A.~Lundervold gave in \cite{HA13} an explicit description of the free post-Lie algebra in terms of formal Lie brackets of planar decorated rooted trees. Let us briefly review this construction: a magma is a set $M$ together with a~binary operation, without any further properties. For any (non-empty) set $E$, the set of all parenthesized words on the alphabet $E$ is the free magma over $E$, denoted $M(E)$. A practical presentation of it can be given in terms of planar rooted trees. Indeed, consider the set $T^{\rm pl}_{E}$ of all planar rooted trees with vertices decorated by $E$, and let $\lbutcher$ denote the left Butcher product defined on~$T^{\rm pl}_{E}$~as
\begin{displaymath}
\sigma \lbutcher \tau=B_{+}^e(\sigma \tau_{1} \tau_{2} \cdots \tau_{k}),
\end{displaymath}
for $\sigma, \tau_{1}, \tau_{2}, \ldots, \tau_{k} \in T^{\rm pl}_{E}$ and $\tau := B^e_{+}(\tau_{1} \tau_{2} \cdots \tau_{k})$. Here, $B_{+}^e$ is the operation defined by grafting a~monomial $\tau_{1} \tau_{2} \cdots \tau_{k}$ of $E$-decorated rooted trees on a common root decorated by some element~$e$ in~$E$, to obtain a new tree. For example (in the undecorated context)
\[
\includegraphics{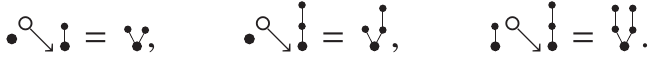}
\]
Denote by $\Cal{T}^{\rm pl}_{E}$ the linear span of the set $T^{\rm pl}_{E}$. Besides the left Butcher product, $\lbutcher$, this space has another magmatic product defined through left grafting, denoted $\searrow$ and defined by
\begin{equation}
\label{SA}
\sigma \searrow \tau = \sum_{v \text{ vertex of }\tau}{\sigma \searrow_v \tau,}
\end{equation}
where $\sigma \searrow_v \tau$ is the tree obtained by grafting the root of the tree $\sigma$ onto the vertex $v$ of the tree~$\tau$, such that $\sigma$ becomes the leftmost branch starting from vertex $v$. See for example references \cite{AM2014, BU72}. Computing some examples (in the undecorated context) we find
\[
\includegraphics{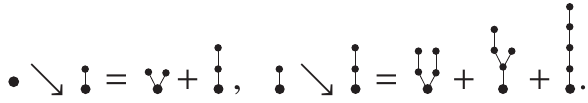}
\]
By freeness universal property, there is a unique morphism of magmatic algebras
\begin{align*}
\Psi\colon\ \big(\Cal{T}^{\rm pl}_{E},\lbutcher\big)&\longrightarrow
 \big(\Cal{T}^{\rm pl}_{E},\searrow\big),
\\
\tau &\longmapsto e_\tau:=\Psi(\tau),
\end{align*}
such that $\Psi(\bullet_a)= \bullet_a$ for any $a\in E$, which is a linear isomorphism.
A detailed account of the map $\Psi$ can be found in~\cite{AM2014}.

Let $\Cal{L}\big(\Cal{T}^{\rm pl}_{E}\big)$ be the free Lie algebra generated by $\Cal{T}^{\rm pl}_{E}$. It can be endowed with a structure of post-Lie algebra by extending the aforementioned left grafting, $\searrow$, as follows
\begin{gather*}
\sigma \searrow [ \tau, \tau']
 = [ \sigma \searrow \tau, \tau'] + [ \tau, \sigma \searrow \tau'],
\\
[\sigma,\tau] \searrow \tau'
 = a_{\searrow}(\sigma, \tau, \tau') - a_{\searrow}(\tau, \sigma, \tau'),
\end{gather*}
for all $\sigma, \tau, \tau' \in \Cal{T}^{\rm pl}_E$. The triple $\big(\Cal{L}\big(\Cal{T}^{\rm pl}_E\big), [\cdot\,,\cdot ], \searrow\!\big)$ is the free post-Lie algebra generated by~$E$~\cite{BV2007} (see also \cite{HA13,MKW08}).

Recall that an $E$-decorated planar forest $f=\tau_1 \cdots \tau_n$ is a (non-commutative) product of $E$-decorated planar rooted trees $\tau_i\in T_E^{\rm pl}$, $i=1,\ldots ,n$. Denote by $F^{\rm pl}_{E}$ the set of all $E$-decorated planar forests, and by $\Cal{F}^{\rm pl}_{E}$ its linear span. The space $\Cal{F}^{\rm pl}_{E}$ forms together with the concatenation product the free associative algebra generated by $\Cal{T}^{\rm pl}_{E}$. The left grafting, $\searrow$, defined by \eqref{SA} on~$\Cal{T}^{\rm pl}_{E}$ can be generalized to a grafting of forests as follows:
\begin{itemize}\itemsep=0pt
\item Left grafting a tree on a forest is also defined by \eqref{SA}. We thus have
\begin{displaymath}
\sigma\searrow{ff'}=(\sigma\searrow f)f'+f(\sigma\searrow f'),
\end{displaymath}
for any tree $\sigma\in T_E^{\rm pl}$ and any two forests $f,f'\in F_E^{\rm pl}$.

\item The left grafting of a forest $f =\tau_1\cdots\tau_k$ onto a forest $f'$ is the sum of forests obtained by summing over all ways of successively left grafting the trees $\tau_k,\ldots,\tau_1$ to any node of $f'$.
\end{itemize}
The well-known (planar) Grossman--Larson product on $\Cal{F}^{\rm pl}_{E}$ is defined by \cite{GL1989}
\begin{equation}
\label{glf}
f \star f' := B_{-}\big(f \searrow B^e_{+}(f')\big),
\end{equation}
where $B_{-}$ is the left inverse operation of $B^e_{+}$, which removes the root of a tree and thus produces a forest. This product endows the space $\Cal{F}^{\rm pl}_{E}$ with a structure of an non-commutative associative unital algebra, called the Grossman--Larson algebra. This algebra acts naturally on $\Cal{T}^{\rm pl}_{E}$ by extended left grafting
\begin{equation*}
(f \star f') \searrow \tau := f \searrow (f' \searrow \tau),
\end{equation*}
for all $f,f'\in \Cal F_E^{\rm pl}$ and $\tau\in \Cal T_E^{\rm pl}$. The universal enveloping algebra $\Cal{U}\big(\Cal{L}\big(\Cal{T}^{\rm pl}_{E}\big)\big)$ of the free post-Lie algebra $\Cal{L}\big(\Cal{T}^{\rm pl}_E\big)$ is the free associative algebra on $\Cal{T}^{\rm pl}_{E}$, and can therefore be identified with $\Cal{F}^{\rm pl}_{E}$. The terminology is justified by the following

\begin{Proposition}[{\cite[Proposition 3.5]{ELM2015}}]
With the identification recalled above, the Grossman--Larson product $\ast$ on $\Cal{U}\big(\Cal{L}\big(\Cal{T}^{\rm pl}_{E}\big)\big)$
is identical to the Grossman--Larson product $\star$ on $\big(\Cal{F}^{\rm pl}_{E}, \star\big)$.
\end{Proposition}

\subsection{Post-Lie structure on a Rota--Baxter algebra}

Recall that a unital algebra is said to be complete filtered if it is equipped with a separating complete filtration
\begin{displaymath}
\Cal A= \Cal{A}_0 \supseteq \Cal{A}_1 \supseteq
\Cal{A}_2 \supseteq \cdots \supseteq \Cal{A}_n \supseteq\cdots
\end{displaymath}
by ideals \cite{KLK04}. Separation means that the intersection of the $\Cal A_n$'s is equal to $\{0\}$, and completeness refers to the topology associated with the filtration, so that any series $a=\sum_{n\ge 1}a_n$ with $a_n\in\Cal A_n$ converges in $\Cal A$. The filtration is moreover supposed to be compatible with the product, i.e., $\Cal{A}_{n_1} \Cal{A}_{n_2}\subseteq \Cal{A}_{n_1 + n_2}$ for any $n_1,n_2\ge 0$. A Rota--Baxter algebra is said to be complete filtered if it is equipped with a separating complete filtration by Rota--Baxter ideals, i.e., by ideals $\Cal{A}_n$ stable by the Rota--Baxter operator.

The Rota--Baxter Algebra $\big(\Cal{A}, \Cal{R}\big)$ of weight $\lambda$ has a structure of a post-Lie algebra defined by the following operations:
\begin{gather}
[x, y]_{\lambda} := \lambda [x, y], \label{pl1}
\\
x \rhd y := [\Cal{R}(x), y] \label{pl2},
\end{gather}
for all $x, y \in \Cal{A}$. We leave it to the reader to show that the operations in \eqref{pl1}, \eqref{pl2} satisfy the post-Lie identities \eqref{eq1} and \eqref{eq2} (see \cite[Section~5.2]{BGN10}). As expected, the post-Lie algebra $\big(\Cal{A}, [\cdot\,, \cdot]_{\lambda}, \rhd \big)$ reduces to a (left) pre-Lie algebra in the case of a weight zero Rota--Baxter algebra. Indeed, if $\lambda =0$, the product $\rhd$ defined by~\eqref{pl2} verifies the left pre-Lie identity~\eqref{left-pl}.

\section{Baker--Campbell--Hausdorff (BCH)-recursion}\label{s2}

We give here a brief account of the Baker--Campbell--Hausdorff recursion, which was defined and explored in \cite{KLK04, KLK05, KLM06}. Let $\Cal{A} = K\left\langle\!\left\langle x, y \right\rangle\!\right\rangle$ be the free complete associative $K$-algebra of formal power series generated by non-commuting variables $x$ and $y$. The Baker--Campbell--Hausdorff expansion ${\rm BCH}(x, y)$ is the element in $\Cal{A}$ satisfying the following equation
\begin{displaymath}
\exp(x)\exp(y) = \exp\big({\rm BCH}(x, y)\big).
\end{displaymath}
The first terms are given by
\begin{align*}
{\rm BCH}(x, y)
&= x+y+\wt{{\rm BCH}}(x,y)\\
&=x+y+\frac{1}{2} [x, y] + \frac{1}{12}[x, [x, y]] - \frac{1}{12}[y, [x, y]]
-\frac{1}{24} [x, [y, [x, y]]] + \cdots,
\end{align*}
where $[x, y] := xy-yx$ is the usual commutator of $x$ and $y$ in $\Cal{A}$. See, e.g., \cite{BF12} for details.

\begin{Proposition}[{\cite[Proposition $1$]{KLM06}}]
Let $\Cal{A}$ be a complete filtered $K$-algebra, and let $\Cal{R}$ be a $K$-linear map preserving the filtration of $\Cal{A}$. There exists a unique $($usually non-linear$)$ map $\Cal{\chi}\colon \Cal{A}_1 \longrightarrow \Cal{A}_1$, such that $(\chi - {\rm id}_{\Cal{A}})(\Cal{A}_n) \subset \Cal{A}_{2n}$, for all $n \geq 1$, and
\begin{equation}
\label{result1}
{\rm BCH}\big(\Cal{R}(\chi(x)), \widetilde{\Cal{R}}(\chi(x))\big) = x,
\end{equation}
for all $x \in \Cal{A}_1$, where $\widetilde{\Cal{R}} := {\rm id}_{\Cal{A}} - \Cal{R}$. This map is bijective, and its inverse is
\begin{equation}
\label{IBCH-recursion}
\chi^{-1}(x)
= {\rm BCH}\big(\Cal{R}(x), \widetilde{\Cal{R}}(x)\big)
= x + \wt{{\rm BCH}}\big(\Cal{R}(x), \widetilde{\Cal{R}}(x)\big).
\end{equation}
\end{Proposition}

As a consequence of \eqref{result1}, we have the exponential factorization
\begin{equation}\label{exp-fact}
\exp\big(\Cal{R}(\chi(x))\big)\exp\big(\widetilde{\Cal{R}}(\chi(x))\big) = \exp(x),
\end{equation}
for any $x \in \Cal{A}_1$. Note also that \eqref{IBCH-recursion} yields the non-linear BCH-recursion
\begin{equation}
\label{BCH-recursion}
\chi(x) := x - \wt{{\rm BCH}}\big(\Cal{R}(\chi(x)), \widetilde{\Cal{R}}(\chi(x))\big),
\end{equation}
for all $x \in \Cal{A}_1$.

\begin{Lemma}[\cite{KLM06}]
Let $\Cal{A}$ be a complete filtered algebra, and let $\Cal{R}\colon \Cal{A} \longrightarrow \Cal{A}$ be a linear map preserving the filtration. The following holds:
\begin{enumerate}\itemsep=0pt
\item[$1.$] The map $\chi$ given by \eqref{BCH-recursion}, can be simplified:
\begin{equation}\label{s1-recursion}
\chi(x) = x + \wt{{\rm BCH}}\big({-}\Cal{R}(\chi(x)), x \big) \qquad \forall x \in \Cal{A}_1.
\end{equation}
\item[$2.$] If $\Cal{R}$ is an idempotent algebra homomorphism, then the map $\chi$ in \eqref{s1-recursion} is further simplified, namely $\chi(x) = x + \wt{{\rm BCH}}(-\Cal{R}(x), x)$.
\end{enumerate}
\end{Lemma}

\begin{proof}
See \cite[Lemmas 6 and 7]{KLM06}.
\end{proof}

The BCH-recursion in the Rota--Baxter algebra framework is given as follows in the case where the weight $\lambda$ is different from zero:

\begin{Proposition}[{\cite[Proposition 11]{KLM06}}]
Let $\big(\Cal{A}, \Cal{R}\big)$ be a complete filtered Rota--Baxter algebra of weight $\lambda\neq 0$, and set $\widetilde{\Cal{R}} := -\lambda {\rm id}_{\Cal{A}}- \Cal{R}$. The $\lambda$-weighted BCH-recursion is written
\begin{equation}
\label{chi-lambda-rec}
\chi_{\lambda}(x)
= x + \frac{1}{\lambda} \wt{{\rm BCH}}\big(\Cal{R}(\chi_{\lambda}(x)),
\widetilde{\Cal{R}}(\chi_{\lambda}(x))\big),
\end{equation}
for all $x\in\Cal A_1$. It can be simplified to
\begin{equation*}
\chi_{\lambda}(x)
= x - \frac{1}{\lambda} \wt{{\rm BCH}}\big({-}\Cal{R}(\chi_{\lambda}(x)), \lambda x\big).
\end{equation*}
Its inverse is given by
\begin{equation*}
\chi_{\lambda}^{-1}(x)
= x - \frac{1}{\lambda} \wt{{\rm BCH}}\big(\Cal{R}(x), \widetilde{\Cal{R}}(x)\big).
\end{equation*}
Moreover, the factorization obtained in \eqref{exp-fact} becomes
\begin{equation}
\label{exp-moins-lambda-x}
\exp\big(\Cal{R}(\chi_{\lambda}(x))\big)\exp\big(\widetilde{\Cal{R}}
(\chi_{\lambda}(x))\big) = \exp(-\lambda x).
\end{equation}
\end{Proposition}

The expansion $\chi_{\lambda}$ can be written as the infinite sum
\begin{equation*}
\chi^{}_{\lambda}(x) = \sum_{n \ge 1}^{}{\chi_{\lambda}^{(n)}(x)},
\end{equation*}
where $\chi_{\lambda}^{(n)}\in\Cal A_n$ is the $n$-th homogenous component of the BCH-recursion. Here, we write the components $\chi_{\lambda}^{(n)}$ up to order $n =4$ using the post-Lie algebra notation~\eqref{pl1} and~\eqref{pl2}
\begin{gather*}
\chi_{\lambda}^{(1)}(x) = x,
\\
\chi_{\lambda}^{(2)}(x) = \frac{1}{2\lambda} \big[\Cal{R}\big(\chi_{\lambda}^{(1)}(x)\big), \widetilde{\Cal{R}}\big(\chi_{\lambda}^{(1)}(x)\big)\big]
\\ \hphantom{\chi_{\lambda}^{(2)}(x) }
{}= \frac{1}{2\lambda} \big[\Cal{R}\big(\chi_{\lambda}^{(1)}(x)\big), \big({-}\lambda {\rm id}_{\Cal{A}}
- \Cal{R} \big)\big(\chi_{\lambda}^{(1)}(x)\big)\big]
 = - \frac{1}{2} \big[\Cal{R}(x), x\big]= - \frac{1}{2} x \rhd x,
 \\
\chi_{\lambda}^{(3)}(x)
= \frac{1}{2 \lambda} \big(\big[\Cal{R}\big(\chi_{\lambda}^{(1)}(x)\big), \widetilde{\Cal{R}}\big(\chi_{\lambda}^{(2)}(x)\big)\big]
+ \big[\Cal{R}\big(\chi_{\lambda}^{(2)}(x)\big), \widetilde{\Cal{R}}\big(\chi_{\lambda}^{(1)}(x)\big)\big]\big)
\\ \hphantom{\chi_{\lambda}^{(3)}(x)=}
{}+ \frac{1}{12 \lambda}\big(\big[\Cal{R}\big(\chi_{\lambda}^{(1)}(x)\big),\big[\Cal{R}\big(\chi_{\lambda}^{(1)}(x)\big), \widetilde{\Cal{R}}\big(\chi_{\lambda}^{(1)}(x)\big)\big]\big]
\\ \hphantom{\chi_{\lambda}^{(3)}(x)=}
{}- \big[\widetilde{\Cal{R}}\big(\chi_{\lambda}^{(1)}(x)\big),\big[\Cal{R}\big(\chi_{\lambda}^{(1)}(x)\big), \widetilde{\Cal{R}}\big(\chi_{\lambda}^{(1)}(x)\big)\big]\big]\big)
\\ \hphantom{\chi_{\lambda}^{(3)}(x)}
{} = \frac{1}{4} ( x \rhd x) \rhd x + \frac{1}{12}x \rhd (x \rhd x) + \frac{1}{12} [x\rhd x, x]_\lambda,
 \\
\chi_{\lambda}^{(4)}(x)= \frac{\lambda-1}{24} x \rhd \big((x \rhd x) \rhd x\big)
- \frac{\lambda + 1}{24} (x \rhd x) \rhd (x \rhd x)
+ \frac{\lambda -3}{24} \big((x \rhd x) \rhd x\big) \rhd x
\\ \hphantom{\chi_{\lambda}^{(4)}(x)=}
{}-\frac{\lambda + 1}{24} \big(x \rhd (x \rhd x)\big) \rhd x
+ \frac{1}{24} \big[x, x \rhd (x \rhd x) + (x \rhd x) \rhd x\big]_\lambda.
\end{gather*}
These coefficients are recursively computed using \eqref{BCH-recursion}.

\section{Magnus expansion}\label{s3}

W.~Magnus \cite{WM54} considered the problem of expressing the solution of the matrix-valued linear initial value problem $\dot Y(t) = M(t)Y(t)$, $Y(0)=Y_0$ as an exponential \cite{BCOR08,MP1970}
\begin{displaymath}
Y(t)=\exp\big({\Omega}(M)(t)\big)Y_0.
\end{displaymath}
The Magnus expansion, ${\Omega}(M)(t)=\log (Y(t))$, is determined by the particular differential equation
\begin{align}
\dot{\Omega}(M) &:=M + \sum_{n>0} \frac{B_n}{n!} \operatorname{ad}_{{\Omega}(M)}^{(n)}(M) \label{MagnusDE1}\\
 &= \operatorname{dexp}^{-1}_{\Omega(M)}(M) \nonumber\\
 &:= \frac{\operatorname{ad}_{\Omega(M)}}{{\rm e}^{\operatorname{ad}_{\Omega(M)}} - 1}(M) \label{MagnusDE2},
\end{align}
with $\Omega(M)(0)=0$. Here, $B_n$ are the Bernoulli numbers and $\operatorname{ad}_{M_1}^{(n)}(M_2):=\operatorname{ad}_{M_1}^{(n-1)}([M_1,M_2])$, $\operatorname{ad}_{M_1}^{(0)}(M_2)=M_2$. Defining the pre-Lie product, $(M_1 \rhd M_2)(t):=\big[\int_0^tM_1(s){\rm d}s,M_2(t)\big]$, we can rewrite \eqref{MagnusDE2} using the left-multiplication operators $L_\rhd(x):=x\rhd -$ defined in terms of the pre-Lie product:
\begin{displaymath}
\dot{\Omega}(M) = \frac{L_\rhd[{\dot{\Omega}(M)}]}{{\rm e}^{L_\rhd[{\dot{\Omega}(M)}]} - 1}(M).
\end{displaymath}

\subsection{Post-Lie Magnus expansion}

We consider now the universal enveloping algebra $\Cal{F}^{\rm pl}_E := \Cal{U}\big(\Cal{L}\big(\Cal{T}^{\rm pl}_E\big)\big)$ of the free post-Lie algebra $\big(\Cal{L}(\Cal{T}^{\rm pl}_E), [\cdot\,,\cdot ], \searrow\! \big)$, graded by the number of vertices of the forests. Denote by $\widehat{\Cal{U}\big(\Cal{L}\big(\Cal{T}^{\rm pl}_E\big)\big)}$ its completion with respect to the grading. Any element of the completion can be written as a~so-called \textit{Lie--Butcher series} \cite{ELM2015,HA13,MKW08}
\begin{equation*}
\alpha = \sum_{f \in F^{\rm pl}_E}^{}{\left\langle \alpha, f \right\rangle f},
\end{equation*}
where $\left\langle \cdot\, , \cdot \right\rangle\! \colon \widehat{\Cal{U}\big(\Cal{L}\big(\Cal{T}^{\rm pl}_E\big)\big)} \otimes \Cal{U}\big(\Cal{L}\big(\Cal{T}^{\rm pl}_E\big)\big) \rightarrow K$ is the natural pairing defined on any pair $(f,f')$ of forests by
\begin{equation*}
\langle f , f' \rangle =
\begin{cases}
0,& f \neq f', \\ 1,& f = f'.
\end{cases}
\end{equation*}
The unshuffle coproduct, $\Delta$, is naturally extended to the completion. The set $\mop{Prim}\big(F^{\rm pl}_E\big)$ consists in primitive elements (\textit{infinitesimal characters}), whereas $G\big(F^{\rm pl}_E\big)$ denotes the set of group-like elements (\textit{characters})
\begin{gather*}
\mop{Prim}(F^{\rm pl}_E) := \big\{\alpha \in \widehat{\Cal{U}\big(\Cal{L}\big(\Cal{T}^{\rm pl}_E\big)\big)}\mid
\Delta(\alpha)
= \mathbf{1} \otimes \alpha + \alpha \otimes \mathbf{1}\big\}=\widehat{\Cal{L}\big(\Cal{T}^{\rm pl}_E\big)},
\\
G(F^{\rm pl}_E) := \big\{\alpha \in \widehat{\Cal{U}\big(\Cal{L}\big(\Cal{T}^{\rm pl}_E\big)\big)}\mid \Delta(\alpha)
= \alpha \otimes \alpha \big\}.
\end{gather*}
Both products on $\Cal{U}\big(\Cal{L}\big(\Cal{T}^{\rm pl}_E\big)\big)$ -- the concatenation and the Grossman--Larson product \eqref{glf} -- can also be extended to products on the completion $\widehat{\Cal{U}\big(\Cal{L}\big(\Cal{T}^{\rm pl}_E\big)\big)}$. As a result, two different exponential functions can be defined on $\widehat{\Cal{U}\big(\Cal{L}\big(\Cal{T}^{\rm pl}_E\big)\big)}$, namely,
\begin{gather*}
\exp^{*}(f) = \sum_{n=0}^{\infty}{\frac{f^{*n}}{n!}}
= \mathbf{1} + f + \frac{1}{2} f \ast f + \frac{1}{6} f \ast f \ast f + \cdots,
\\
\exp(f) = \sum_{n=0}^{\infty}{\frac{f^{n}}{n!}}
= \mathbf{1} + f + \frac{1}{2} f f + \frac{1}{6} f f f + \cdots.
\end{gather*}
Both these exponential functions map $\mop{Prim}\big(F^{\rm pl}_E\big)$ bijectively onto $G\big(F^{\rm pl}_E\big)$. See \cite{ELM2015} for details. The post-Lie Magnus expansion $\chi$ is the bijective map from $\widehat{\Cal{L}\big(\Cal{T}^{\rm pl}_E\big)}$ onto itself defined by
\begin{displaymath}
\exp^{*}\big(\chi(f)\big) = \exp(f),
\end{displaymath}
namely,
\begin{equation}
\label{pl-Magnus}
\chi(f) = \log^{\ast}\big(\exp(f)\big).
\end{equation}
Introducing a formal commuting indeterminate $t$, it can also be described as
\[
\chi(ft) = \sum_{n\geq 1}{\chi^{(n)}(f)}t^n,
\]
where $\chi^{(n)}(f)$ is the $n$-th order component of the post-Lie Magnus expansion $\chi$. The latter is defined recursively by $\chi^{(1)}(f) = f$, and \cite{ELM2015, EI2018, EIM2017}
\begin{equation}\label{chi-rec}
\chi^{(n)}(f)
:= \frac{f^{n}}{n!} - \sum_{k=2}^{n}{\frac{1}{k!}\sum_{\substack{p_{1} + \cdots + p_{k}
= n \\ p_{i}>0}}{\chi^{(p_{1})}(f)\ast \chi^{(p_{2})}(f) \ast \cdots \ast \chi^{(p_{k})}(f)}}.
\end{equation}
The computation of the coefficients $\chi^{(n)}(f)$ for the first five values of $n$ is displayed in Appendix~\ref{C.pl-Magnus} below. They have been obtained by hand by the recursive formula \eqref{chi-rec}, using \eqref{GL-simple} repeatedly. Comparing with the computations at the end of Section \ref{s2}, one observes that, up to order $n=4$, the coefficient $\chi^{(n)}(f)$ coincides with the coefficient $\chi_\lambda^{(n)}(f)$ of the BCH-recursion in the weight $\lambda = 1$ case. We shall prove this fact at any order in Theorem \ref{main Theorem} below.

\begin{Remark}\label{rem1}
Formula \eqref{pl-Magnus} defines the post-Lie Magnus expansion in any complete filtered post-Lie algebra $L$. If the underlying Lie algebra is Abelian, then the post-Lie Magnus expansion is reduced to the so-called pre-Lie Magnus expansion. The latter already appears in \cite{AG81} and encompasses classical Magnus expansion \cite{WM54}.
\end{Remark}

\begin{Remark}\label{remMagnus1}
 We may deduce a Magnus-type differential equation similar to \eqref{MagnusDE1} for the post-Lie Magnus expansion \eqref{pl-Magnus}, by differentiating $\exp^{*}\big(\chi(ft)\big) = \exp(ft)$ with respect to~$t$. This results in
\begin{displaymath}
\dot{\chi}(ft) = \operatorname{dexp}^{\ast -1}_{- \chi(ft)}\big( \exp^*(- \chi(ft)) \rhd f \big), \qquad \chi(0)=0.
\end{displaymath}
\end{Remark}

\subsection{Inverse post-Lie Magnus expansion}

The inverse post-Lie Magnus expansion $\theta$ is the bijective map from $\widehat{\Cal{L}\big(\Cal{T}^{\rm pl}_E\big)}$ onto itself given by the following formula
\begin{equation}\label{IPLME}
\theta(f) = \log(\exp^*(f))
\end{equation}
or
\begin{equation*}
\exp(\theta(f)) = \exp^*(f).
\end{equation*}
The homogeneous component $\theta^{(n)}=\theta^{(n)}(f)$ of degree $n$ of the expansion
\begin{displaymath}
\theta(ft)=\sum_{n\geq 1}{\theta^{(n)}(f)t^n}
\end{displaymath}
is given by $\theta^{(1)}(f)=f$ and the following recursive formula \cite{ELM2015}
\begin{gather}
\theta^{(n)}(f)
= \frac{1}{n} \Bigg( \sum_{j=1}^{n-1}{\frac{1}{j!}\sum_{\substack{k_1 + \cdots + k_j = n-1 \\ k_i>0}}
{\big(\theta^{(k_1)}\theta^{(k_2)} \cdots \theta^{(k_j)}\big) \rhd f}}\nonumber
\\ \hphantom{\theta^{(n)}(f)=\frac{1}{n} \Bigg(}
+ \sum_{j=1}^{n-1}{\frac{B_j}{j!}\sum_{\substack{k_1 + \cdots + k_j = n-1 \\ k_i>0}}
{\operatorname{ad}_{\theta^{(k_1)}} \cdots \operatorname{ad}_{\theta^{(k_j)}} f}}\nonumber
\\ \hphantom{\theta^{(n)}(f)=\frac{1}{n} \Bigg(}
{}+ \sum_{j=2}^{n-1} \Bigg(\Bigg( \sum_{q=1}^{j-1}{\frac{B_q}{q!}
\sum_{\substack{k_1 + \cdots + k_q = j-1 \\ k_i>0}}\operatorname{ad}_{\theta^{(k_1)}} \cdots \operatorname{ad}_{\theta^{(k_q)}} \Bigg)}\nonumber
\\ \hphantom{\theta^{(n)}(f)=\frac{1}{n} \Bigg(}
{}\times\Bigg(\sum_{p=1}^{n-j}{\frac{1}{p!}\sum_{\substack{k_1 + \cdots + k_p = n-j \\
k_i>0}}{\big(\theta^{(k_1)}\theta^{(k_2)} \cdots \theta^{(k_p)}\big) \rhd f}} \Bigg)\Bigg)\Bigg),
\label{IPL}
\end{gather}
where $\operatorname{ad}_{\theta^{(i)}}(f):= \big[\theta^{(i)}, f\big]$, and the $B_i$'s are the Bernoulli numbers. The computation of the first~$\theta^{(n)}$'s is given in Appendix \ref{C.ipl-Magnus}.

\begin{Remark}
The same fact, described in Remark \ref{rem1}, will be repeated again in the case of the inverse post-Lie Magnus expansion. In other words, the formula in \eqref{IPL} for the inverse post-Lie Magnus expansion is reduced, in the case of commutative post-Lie algebras, to the inverse pre-Lie Magnus expansion formula described below (see also \cite{KM08,AL2011})
\begin{gather*}
W(x) := \frac{{\rm e}^{L_{\rhd}[x]} - 1}{L_{\rhd}[x]}(x) = \sum_{n=0}^{\infty}\frac{1}{(n+1)!} L^{(n)}_{\rhd}[x](x).
\end{gather*}
Modulo removal of a fictitious unit, $W$ is also known as the pre-Lie exponential \cite{AG81}.
\end{Remark}

\begin{Remark}
Similar to Remark \ref{remMagnus1}, we may deduce a Magnus-type differential equation similar to \eqref{MagnusDE1} for the inverse post-Lie Magnus expansion \cite{ELM2015,EIM2017}
\begin{displaymath}
\dot{\theta}(ft) = \operatorname{dexp}^{-1}_{- \theta(ft)}\big( \exp(\theta(ft)) \rhd f \big), \qquad \theta(0)=0.
\end{displaymath}
\end{Remark}

\section{Post-Lie Magnus expansion and BCH-recursion}\label{s4}

We now show that the Baker--Campbell--Hausdorff recursion driven by a weight $\lambda=1$ Rota--Baxter operator identifies with the Magnus expansion relative to the post-Lie structure naturally associated to the corresponding Rota--Baxter algebra.

\begin{Theorem}\label{T9}
Let $\big(\Cal{A}, \Cal{R}\big)$ be a complete filtered Rota--Baxter algebra of weight $\lambda = 1$. We have the following equality in $\widehat{\Cal U(\Cal A)}$ for any $x\in\Cal A$ and $t \in K$:
\begin{equation}\label{ET9}
\exp^{\ast}(tx) = \exp\big({-}t \widetilde{\Cal{R}}(x)\big) \exp\big({-}t \Cal{R}(x)\big),
\end{equation}
where $\widetilde{\Cal{R}}= -{\rm id}_{\Cal{A}}-\Cal{R}$, and $\ast$ is the associative product defined in \eqref{ast}, using the post-Lie product $x \rhd y = [\Cal{R}(x), y]$.
\end{Theorem}
 The proof of this theorem will rely on the following proposition:

\begin{Proposition}
In any complete filtered Rota--Baxter algebra $\big(\Cal{A}, \Cal{R}\big)$ of weight $\lambda = 1$, we have the following identity in $\widehat{\Cal U(\Cal A)}$:
\begin{equation}
\label{diffeq1}
\frac{{\rm d}^n}{{\rm d}t^n} \exp\big({-}t \widetilde{\Cal{R}}(x)\big) \exp\big({-}t \Cal{R}(x)\big)
= \exp\big({-}t \widetilde{\Cal{R}}(x)\big) x^{\ast n} \exp\big({-}t \Cal{R}(x)\big).
\end{equation}
\end{Proposition}

\begin{proof}
The proof goes by induction. The base case $k=0$ is trivial. For $k=1$, we have in~$\widehat{\Cal U(\Cal A)}$:
\begin{align*}
\frac{\rm d}{{\rm d}t} \exp\big({-}t \widetilde{\Cal{R}}(x)\big) \exp\big({-}t \Cal{R}(x)\big)
= {}&\exp\big(t ({\rm id}_{\Cal{A}} + \Cal{R})(x)\big) ({\rm id}_{\Cal{A}} + \Cal{R})(x) \exp\big({-}t \Cal{R}(x)\big)
\\
&-\exp\big(t ({\rm id}_{\Cal{A}} + \Cal{R})(x)\big) \Cal{R}(x) \exp\big({-}t \Cal{R}(x)\big)
\\
={}& \exp\big({-}t \widetilde{\Cal{R}}(x)\big) x \exp\big({-}t \Cal{R}(x)\big).
\end{align*}
Now, suppose that the statement is true in the case $k=n-1$, i.e.,
\begin{equation*}
\frac{{\rm d}^{n-1}}{{\rm d}t^{n-1}} \exp\big({-}t \widetilde{\Cal{R}}(x)\big) \exp\big({-}t \Cal{R}(x)\big)
= \exp\big({-}t \widetilde{\Cal{R}}(x)\big) x^{\ast {n-1}} \exp\big({-}t \Cal{R}(x)\big).
\end{equation*}
{\samepage We then get
\begin{gather*}
\frac{{\rm d}^{n}}{{\rm d}t^{n}} \exp\big({-}t \widetilde{\Cal{R}}(x)\big) \exp\big({-}t \Cal{R}(x)\big)
\\ \qquad
{} = \frac{\rm d}{{\rm d}t} \frac{{\rm d}^{n-1}}{{\rm d}t^{n-1}} \exp\big({-}t \widetilde{\Cal{R}}(x)\big) \exp\big({-}t \Cal{R}(x)\big)
 = \frac{\rm d}{{\rm d}t} \exp\big({-}t \widetilde{\Cal{R}}(x)\big) x^{\ast {n-1}} \exp\big({-}t \Cal{R}(x)\big)
\\ \qquad
{}= \exp\big({-}t \widetilde{\Cal{R}}(x)\big) ({\rm id}_{\Cal{A}} + \Cal{R})(x) x^{\ast n-1}
\exp\big({-}t \Cal{R}(x)\big)
\\ \qquad\hphantom{=}
{} - \exp\big({-}t \widetilde{\Cal{R}}(x)\big) x^{\ast n-1}\Cal{R}(x) \exp\big({-}t \Cal{R}(x)\big)
\\ \qquad
{}= \exp\big({-}t \widetilde{\Cal{R}}(x)\big) \big(x x^{\ast n-1} + \Cal{R}(x) x^{\ast n-1} - x^{\ast n-1} \Cal{R}(x)\big) \exp\big({-}t \Cal{R}(x)\big)
\\ \qquad
{} = \exp\big({-}t \widetilde{\Cal{R}}(x)\big) \big(x x^{\ast n-1} + x \rhd x^{\ast n-1}\big) \exp\big({-}t \Cal{R}(x)\big)
\\ \qquad
{}= \exp\big({-}t \widetilde{\Cal{R}}(x)\big) \big( x \ast x^{\ast n-1}\big) \exp\big({-}t \Cal{R}(x)\big)
 = \exp\big({-}t \widetilde{\Cal{R}}(x)\big) x^{\ast {n}} \exp\big({-}t \Cal{R}(x)\big),
\end{gather*}

}

\noindent which means that \eqref{diffeq1} is true for $k=n$, and it is true for all $n \geq 0$. This ends the proof.
\end{proof}

\begin{proof}[Proof of Theorem \ref{T9}]
We have that
\begin{equation*}
\frac{{\rm d}^n}{{\rm d}t^n} \exp^{\ast}(tx) = x^{\ast n} \exp^{\ast}(tx),
\end{equation*}
thus,
\begin{equation*}
\left.\frac{{\rm d}^n}{{\rm d}t^n}\right\vert_{t=0} \exp^{\ast}(tx) = \left.\frac{{\rm d}^n}{{\rm d}t^n}\right\vert_{t=0} \exp\big({-}t \widetilde{\Cal{R}}(x)\big) \exp\big({-}t \Cal{R}(x)\big)\qquad \hbox{for all}\quad n\geq 0.
\end{equation*}
One can therefore conclude that both members of \eqref{ET9} do coincide as infinite formal series.
\end{proof}
\begin{Theorem}\label{main Theorem}
The post-Lie Magnus expansion $\chi$, described in \eqref{pl-Magnus}, coincides with the weighted BCH-recursion $\chi_{\lambda}$ recursively given by \eqref{chi-lambda-rec}, with weight $\lambda=1$.
\end{Theorem}

\begin{proof}
From equation \eqref{exp-moins-lambda-x}, specialized to $\lambda=1$, and by setting $\theta_{\rm BCH} :=\chi_{1}^{-1}$, we obtain that
\begin{displaymath}
\exp\big({-}\theta_{\rm BCH}(tx)\big)
= \exp\big(t \Cal{R}(x)\big) \exp\big(t \widetilde{\Cal{R}}(x)\big)
\end{displaymath}
in $\widehat{\Cal U(\Cal A)}$, which is equivalent to
\begin{equation}\label{res.2}
\exp\big(\theta_{\rm BCH}(tx)\big)
= \exp\big({-}t \widetilde{\Cal{R}}(x)\big) \exp\big({-}t \Cal{R}(x)\big).
\end{equation}
From \eqref{IPLME}, \eqref{ET9} and \eqref{res.2} we have
\begin{equation}\label{res.3}
\exp\big(\theta_{\rm BCH}(tx)\big) = \exp^{\ast}(tx)= \exp\big(\theta(tx)\big).
\end{equation}
Then the two $\theta$'s, namely the inverse BCH-recursion in~\eqref{res.2} and the inverse post-Lie Magnus expansion \eqref{res.3}, do coincide.
\end{proof}

\appendix

\section{Calculations on post-Lie Magnus expansion}\label{C.pl-Magnus}

The first five elements of the post-Lie Magnus expansion are
\begin{gather*}
\chi^{(1)}(f) = f,
\\
\chi^{(2)}(f) = -\frac{1}{2}f \rhd f,
\\
\chi^{(3)}(f) = \frac{1}{12}f \rhd (f \rhd f) + \frac{1}{4}(f \rhd f) \rhd f + \frac{1}{12}[f \rhd f, f],
\\
\chi^{(4)}(f) = -\frac{1}{12}\big( (f \rhd f) \rhd (f \rhd f)
+ (f \rhd (f \rhd f)) \rhd f+ ((f \rhd f) \rhd f) \rhd f \big)
\\ \hphantom{\chi^{(4)}(f) =}
{} + \frac{1}{24}\big([f, f \rhd (f \rhd f)] + [f, (f \rhd f) \rhd f]\big),
\\
\chi^{(5)}(f) = -\frac{1}{720}f \rhd \big( f \rhd \big(f \rhd (f \rhd f)\big)\big)
\\ \hphantom{\chi^{(5)}(f) = }
{}+ \frac{1}{144}\big((f \rhd f) \rhd \big(f \rhd (f \rhd f)\big) - f \rhd \big(((f \rhd f) \rhd f) \rhd f\big)
\\ \hphantom{\chi^{(5)}(f) = }
{}-\! f \rhd \big((f \rhd (f \rhd f)) \rhd f\big) \!- f \rhd \big(f \rhd ((f \rhd f) \rhd f)\big) \!+ 5\big(f \rhd (f \rhd f)\big) \rhd (f \rhd f)
\\ \hphantom{\chi^{(5)}(f) = }
{} +5\big((f \rhd f) \rhd f\big) \rhd (f \rhd f) + 6\big((f \rhd f) \rhd (f \rhd f)\big) \rhd f
\\ \hphantom{\chi^{(5)}(f) = }
{} + 3\big((f \rhd (f \rhd f)\big) \rhd f ) \rhd f+
 3\big(f \rhd (f \rhd (f \rhd f))\big) \rhd f
\\ \hphantom{\chi^{(5)}(f) = }
{}+ 3\big(f \rhd ((f \rhd f) \rhd f)\big) \rhd f + 3(f \rhd f) \rhd \big((f \rhd f) \rhd f\big)
\\ \hphantom{\chi^{(5)}(f) = }
{}+ 3\big(((f \rhd f) \rhd f) \rhd f\big) \rhd f \big) + \frac{1}{180}\big[f, [f, f \rhd (f \rhd f)] - f \rhd (f \rhd (f \rhd f))\big]
\\ \hphantom{\chi^{(5)}(f) = }
{}- \frac{1}{120} [ f \rhd f, f \rhd (f \rhd f)]- \frac{1}{36}[f, (f \rhd f) \rhd (f \rhd f)]- \frac{1}{72}\big[f, f \rhd \big((f \rhd f) \rhd f\big)
\\ \hphantom{\chi^{(5)}(f) = }
{}+ \big(f \rhd (f \rhd f)\big) \rhd f + \big((f \rhd f) \rhd f\big) \rhd f\big]-\frac{1}{360}\big[ f \rhd f, [f, f \rhd f]\big]
\\ \hphantom{\chi^{(5)}(f) = }
{} + \frac{1}{720}\big[f, \big[f, [f, f \rhd f]\big]\big].
\end{gather*}

\section{Computations on the inverse post-Lie Magnus expansion}
\label{C.ipl-Magnus}
Here, we calculate the first five inverse post-Lie Magnus elements:
\begin{gather*}
\theta^{(1)}(f) = f,
\\
\theta^{(2)}(f) = \frac{1}{2} f \rhd f ,
\\
\theta^{(3)}(f) = \frac{1}{6} f \rhd (f \rhd f) + \frac{1}{12} [f, f \rhd f] ,
\\
\theta^{(4)}(f) = \frac{1}{24} \big( f \rhd \big(f \rhd (f \rhd f)\big) + [f, f \rhd (f \rhd f)] \big),
\\
\theta^{(5)}(f) = \frac{1}{120} f \rhd \big( f \rhd \big(f \rhd (f \rhd f)\big)\big) + \frac{1}{80} \big[f, f \rhd \big(f \rhd (f \rhd f)\big)\big]
\\ \hphantom{\theta^{(5)}(f) =}
{}+ \frac{1}{720} \big( \big[f, [f, f \rhd (f \rhd f)]\big] - \big[f, \big[f, [f, f \rhd f]\big]\big]\big)
\\ \hphantom{\theta^{(5)}(f) =}
{}+ \frac{1}{120} [f \rhd f, f \rhd (f \rhd f)] - \frac{1}{240} \big[f \rhd f, [f, f \rhd f]\big].
\end{gather*}

\subsection*{Acknowledgments} The first author was funded by the Iraqi Ministry of Higher Education and Scientific Research. He would like to thank the Department of Mathematics at the University of Bergen, Norway, for warm hospitality during a visit in 2021, which was partially supported by the project Pure Mathematics in Norway, funded by Trond Mohn Foundation and Troms{\o} Research Foundation.
He would also like to thank Mustansiriyah University, College of Science, Mathematics Department for their support in carrying out this work. The second author is supported by the Research Council of Norway through project 302831 ``Computational Dynamics and Stochastics on Manifolds'' (CODYSMA). The third author is supported by Agence Nationale de la Recherche, projet CARPLO ANR20-CE40-0007. We thank the anonymous referees for their remarks and suggestions which greatly helped us to improve the presentation of the present paper.

\pdfbookmark[1]{References}{ref}
\LastPageEnding

\end{document}